\documentclass[11pt]{amsart}

\usepackage{amsmath,amssymb,amsthm,mathtools}
\mathtoolsset{showonlyrefs}
\usepackage[margin=1.15in]{geometry}
\usepackage[colorlinks=true,citecolor=blue,linkcolor=blue,urlcolor=blue]{hyperref}

\numberwithin{equation}{section}

\newtheorem{theorem}{Theorem}[section]
\newtheorem{proposition}[theorem]{Proposition}
\newtheorem{lemma}[theorem]{Lemma}

\theoremstyle{definition}

\theoremstyle{plain}
\newtheorem{remark}[theorem]{Remark}
\newcommand{\R}{\mathbb R}

\DeclareMathOperator{\Lip}{Lip}
\DeclareMathOperator{\osc}{osc}
\DeclareMathOperator{\tr}{tr}
\DeclareMathOperator{\cof}{cof}

\title[Sharp Small-Height Lipschitz Estimates]
{Sharp Small-Height Estimates for Subcritical Special Lagrangian Equations in Dimension Three}

\begin{document}

\begin{abstract}
We prove an interior Lipschitz estimate with a sharp height threshold for continuous viscosity solutions of the three-dimensional special Lagrangian equation with subcritical phase $|\Theta|<\pi/2$. If
\[
\osc_{B_R}u\leq MR^2,\qquad 0\leq M<\sec\Theta,
\]
then $\Lip_{B_{\rho R}}u\leq CR$ for some $\rho>0$ and $C>0$ depending only on $|\Theta|$ and $M$. The threshold is optimal: at $M=\sec\Theta$, there are entire real-analytic solutions whose Lipschitz seminorms are unbounded on every fixed interior ball. The proof uses a two-point maximum principle and algebraic identities specific to dimension three.
\end{abstract}

\author{Caiyan Li}
\address[C.L]{School of Mathematical Sciences, Xiamen University, 361005, Xiamen, P.R. China}
\email{caiyanli@xmu.edu.cn}
\thanks{C.L. is supported by NSFC (Grant No.~12501274).}
 
\author{Wei Wei}
\address[W.W]{School of Mathematics, Nanjing University, 210093, Nanjing, P.R. China}
\email{wei\_wei@nju.edu.cn} 
 \thanks{W.W. is supported by NSFC (Grant No.~12571218)}
\maketitle

\section{Introduction}

We consider the special Lagrangian equation
\begin{equation}
	\label{eq:main} \sum_{i=1}^{n}\arctan\lambda_i(D^2u)=\Theta
	\quad\text{in } \mathbb R^n,
\end{equation}
where $\lambda_i(D^2u)$ are the eigenvalues of $D^2u$. Equation \eqref{eq:main} originates in the calibrated geometry of Harvey-Lawson \cite{HarveyLawson}. We identify $\mathbb R^n\times\mathbb R^n$ with $\mathbb C^n$ via $z_j=x_j+\mathrm{i}y_j$ and equip it with the standard symplectic form $\omega=\sum_{j=1}^n dx_j\wedge dy_j$ and the holomorphic volume form $\Omega=dz_1\wedge\cdots\wedge dz_n$. For a smooth function $u$, the symmetry of $D^2u$ implies that the gradient graph $\Gamma_u=\{(x,Du(x))\}$ is Lagrangian, and \eqref{eq:main} is exactly the condition that $\Gamma_u$ be special Lagrangian with phase $\Theta$. Hence $\Gamma_u$ is calibrated, and therefore minimal and locally volume-minimizing.

The phase is called \textit{critical} if $|\Theta|=(n-2)\pi/2$, \textit{supercritical} if $|\Theta|>(n-2)\pi/2$, and \textit{subcritical} if $|\Theta|<(n-2)\pi/2$. Yuan \cite{YuanGlobal} showed that the corresponding level set in the space of symmetric $n\times n$ matrices is convex precisely in the critical and supercritical ranges. This convexity fails in the subcritical range, making it fundamentally different from the critical and supercritical cases.

 In the critical and
supercritical ranges, the scale-invariant estimate
\begin{equation}
    \label{eq:known-gradient-intro}
    \max_{B_R}|Du|
    \leq C(n)\frac{\osc_{B_{2R}}u}{R}
\end{equation}
holds \cite{WarrenYuan,YuanLectures}.  Our focus is the fundamentally
different case $n=3$ and $|\Theta|<\pi/2$.

It is natural to ask whether an estimate of the form
\eqref{eq:known-gradient-intro} continues to hold in this subcritical
range.  Fan, the first author, and Wang \cite{FanLiWang} recently answered
this question negatively: for every fixed $|\Theta|<\pi/2$, they
constructed entire real-analytic solutions with uniformly bounded
oscillation on $B_1$ but unbounded Lipschitz seminorm on every smaller
ball.  A particularly informative family is
\begin{equation}
    \label{eq:intro-counterexample}
    U_{k,\Theta}(x,y,z)
    =\frac{\tan\Theta}{2}(x^2+y^2)
    +\frac{\sec\Theta}{2}
    \left[\cos(2kz)(x^2-y^2)-2\sin(2kz)xy\right].
\end{equation}
For every $k$,
\begin{equation}
    \label{counterexample for lips}
    \osc_{B_1}U_{k,\Theta}=\sec\Theta,
\end{equation}
whereas, for every fixed $r\in(0,1)$,
\[
    \Lip_{B_r}U_{k,\Theta}
    \geq \left|\partial_zU_{k,\Theta}(r/2,r/2,0)\right|
    =\frac{kr^2}{2}\sec\Theta
    \longrightarrow\infty.
\]

The family~\eqref{eq:intro-counterexample} contains more information than
the mere failure of an unconditional gradient estimate. 
The counterexamples lie exactly at the level $\osc_{B_1} u=\sec\Theta$.  Thus they do not rule out
gradient control below that level. Instead, the counterexamples suggest the following
sharper question:

\medskip
\begin{quote}
\emph{Does there exist a sharp constant $M_{\mathrm{crit}}(\Theta)$ such that  uniform interior Lipschitz estimates hold precisely if  $\osc_{B_1} u<M_{\mathrm{crit}}(\Theta)$?}
\end{quote}
\medskip

In this formulation the examples above give the upper bound
$M_{\mathrm{crit}}(\Theta)\leq\sec\Theta$.
We address  this problem in this paper. 

There is also evidence that small height is the correct hypothesis.  For
smooth solutions, one can try a one-point maximum-principle argument, which is applied to an
auxiliary function of the form
\begin{equation}
    \label{auxiliary function1}
    G(x)=
    \bigl(1-|x|^2+\beta(u(x)-\sup_{B_1}u)\bigr)^2
    \log\sqrt{1+|Du|^2}
\end{equation}
gives an interior gradient estimate when
$M<\frac12\tan|\Theta|$; at $\Theta=0$ this argument gives no positive
range.  The one-point calculation therefore identifies the right type of
assumption, but it does not reach the height $\sec\Theta$ ruled out by
the counterexamples.  Closing this gap is the problem addressed here.

For $E\subset\mathbb R^n$, set
\[
    \Lip_Eu
    :=\sup_{x,y\in E,\ x\neq y}
    \frac{|u(x)-u(y)|}{|x-y|}.
\]
Our main result determines the sharp height.

\begin{theorem}\label{thm:small-height}
Let $B_R\subset\mathbb R^3$, and let $u\in C(B_R)$ be a viscosity
solution of~\eqref{eq:main} with $|\Theta|<\pi/2$.  Suppose that
\[
    \osc_{B_R}u\leq MR^2,
    \qquad 0\leq M<\sec\Theta.
\]
Then there exist constants $C>0$ and $0<\rho<1$, depending only on
$|\Theta|$ and $M$, such that
\[
    \Lip_{B_{\rho R}}u\leq CR.
\]
\end{theorem}

Consequently,
\[
    M_{\mathrm{crit}}(\Theta)=\sec\Theta.
\]

The condition is sharp: the entire real-analytic family
\eqref{eq:intro-counterexample} has height $M=\sec\Theta$ but unbounded
interior Lipschitz seminorms, ruling out uniform estimates for
$M\geq\sec\Theta$. Thus the small-height assumption is intrinsic, with
its optimal threshold dictated by the counterexamples. The
theorem applies to continuous viscosity solutions without convexity,
semiconvexity, or an a priori $C^{1,1}$ bound.

Let us briefly place the result in the existing regularity theory for special Lagrangian equations.

Warren-Yuan \cite{WarrenYuan} also obtained a Hessian estimate in dimension three, and
Wang-Yuan \cite{WangYuanHessian} subsequently extended it to dimensions $n\ge3$. Together, these gradient and Hessian estimates yield interior analyticity for continuous viscosity solutions. Yuan \cite{YuanBernstein,YuanGlobal} further showed that entire convex solutions  and smooth entire solutions with supercritical phase are quadratic.
Related Hessian estimates were obtained by the first author through compactness \cite{LiCompactness}, by Ding under an additional structural condition  \cite{DingHessian}, and by Shankar through a pointwise
doubling argument \cite{ShankarDoubling}.

Under a convexity assumption, 
Chen-Warren-Yuan \cite{ChenWarrenYuan} obtained interior Hessian estimates for smooth convex solutions, and Chen-Shankar-Yuan \cite{ChenShankarYuan} proved that convex viscosity solutions are real analytic. Without assuming convexity, Yuan \cite{YuanAPriori}  obtained an interior $C^{2,\alpha}$ estimate under an a priori $C^{1,1}$ bound in dimension three. For small $L^\infty$ perturbations of quadratic solutions, Savin \cite[Theorem~1.3]{SavinSmall} proved interior $C^{2,\alpha}$ regularity.

Without these assumptions, the behavior in the subcritical case is markedly different. Nadirashvili-Vl\u{a}du\c{t}  \cite{NadirashviliVladut} constructed a singular $C^{1,1/3}$ viscosity solution in dimension three. For every subcritical phase, Wang-Yuan \cite{WangYuanSingular} constructed singular $C^{1,1/(2m-1)}$ viscosity solutions for each integer $m\ge2$ and showed that interior Hessian estimates in terms of the gradient alone fail. Later, Mooney-Savin  \cite{MooneySavin} constructed a semiconvex Lipschitz solution that is not $C^1$, whose gradient graph is not minimal. Also, Mooney \cite{MooneyHomogeneous} proved that a $C^2$ function on $\mathbb R^n\setminus\{0\}$, homogeneous of degree $\alpha\in(0,1)$ and satisfying $\det D^2u\ne0$, cannot change sign, thereby ruling out counterexamples to interior gradient estimates obtained through this homogeneous construction. For $-(n-2)\pi/2<\Theta<\pi/2$, Mooney-Shankar
\cite{MooneyShankar} proved that viscosity solutions are analytic provided that $D^2u\ge -\tan\bigl((\pi/2-\Theta)/(n-1)\bigr)I$. This semiconvexity constant is sharp: any larger constant
admits Lipschitz viscosity solutions that are not $C^1$.  

The proof is guided by the way the sharp question is formulated.  Both
the hypothesis $\osc u\leq M$ and the desired Lipschitz conclusion compare
the values of $u$ at two points.  We therefore replace the one-point
quantity~\eqref{auxiliary function1} by a two-point function.  After
scaling to $B_1$, its essential form is
\begin{equation}
    \label{eq:Z-intro}
    Z(x,y)=u(y)-u(x)
    - \sec\Theta 
    \left[
        L|y-x|-\frac{L^2|y-x|^2}{4a}
        +\beta\left|\frac{x+y}{2}-z_0\right|^2
    \right].
\end{equation}
The parameters can be chosen so that
\[
    M\cos\Theta<a<\rho^2,
    \qquad \frac{a}{\rho^2}<\beta<1,
\]
precisely when $M\cos\Theta<1$.  The first inequality excludes a
positive maximum of $Z$ on the outer part of the two-point boundary,
while $\beta<1$ is what eventually rules out a positive interior maximum.
Thus the sharp threshold is already visible in the geometry of the
two-point comparison.

The interior argument uses algebra special to the three-dimensional
equation.  With the phase shift
\[
    W=\cos\Theta\,D^2u-\sin\Theta\,I,
\]
equation~\eqref{eq:main} yields
\[
    \sigma_3(W)-\sigma_1(W)=2\sin\Theta,
    \qquad I-\cof W>0.
\]
At a positive interior maximum of $Z$, the second-derivative inequality,
together with these identities, produces two positive-definite
$2\times2$ matrices.  The two-point inequality gives an upper bound for
the determinant of their sum, whereas the phase identities give a
strictly incompatible lower bound.  This contradiction proves the
estimate.  For continuous viscosity solutions, the classical
second-derivative inequality is replaced by the theorem of sums of
Crandall--Ishii--Lions \cite{CrandallIshiiLions}.

Height-controlled gradient estimates go back to
Bombieri--De Giorgi--Miranda \cite{BombieriDeGiorgiMiranda},
Bombieri--Giusti \cite{BombieriGiusti}, and Korevaar \cite{Korevaar}; see
also \cite{WangWeiZhang} for a recent one-point argument in a capillary
setting.  Two-point maximum principles have been effective in viscosity
theory and in several geometric and spectral problems; see, for example,
\cite{CrandallIshiiLions,AndrewsMultiPoint,Andrews-matrix,
AndrewsClutterbuckGap,BrendleLawson}.  The new feature here is that the
two-point method, combined with the phase-shift identities above, detects
the \emph{exact} height at which subcritical gradient control breaks down.

\medskip

\section{Proof of the main theorem}

Let $\operatorname{Sym}(n)$ denote the space of real symmetric $n\times n$ matrices. For a matrix $A\in\mathbb R^{n\times n}$, let $\cof A$ denote its cofactor matrix. If $A^{ij}$
is obtained from $A$ by deleting its $i$th row and $j$th column, then
\[
(\cof A)_{ij}=(-1)^{i+j}\det A^{ij}.
\]
For every $Q\in\operatorname{Sym}(2)$, the definition of the cofactor matrix gives
\begin{equation}\label{cof A linear in 2-2}
\cof Q=(\tr Q)I_2-Q,
\end{equation}
where $I_2$ is the identity matrix.

For the  readers' convenience, we state the basic algebraic facts used in the proof.

\begin{lemma}\label{lem:phase-identity}
Let $H\in\operatorname{Sym}(3)$  satisfy 
\[
\sum_{i=1}^{3}\arctan\lambda_i(H)=\Theta,
\qquad
|\Theta|<\frac{\pi}{2}.
\]
Set
\[
W:=\cos\Theta\,H-\sin\Theta\,I.
\]
Then
\begin{align}
\sigma_3(W) -\sigma_1(W)&=2\sin\Theta,
\label{eq:phase-det-trace}\\
I-\cof W&>0.
\label{eq:phase-cofactor}
\end{align}
\end{lemma}

\begin{proof}
Denote $\lambda_i=\lambda_i(H)$ and $\theta_i=\arctan\lambda_i.$
Then $\theta_1+\theta_2+\theta_3=\Theta$. 
Since  
\[
\prod_{i=1}^{3}(1+\mathrm{i}\lambda_i)
=
e^{\mathrm{i}\Theta}
\prod_{i=1}^{3}\sqrt{1+\lambda_i^2},
\]
taking real and imaginary parts gives
\begin{align}
1-\sigma_2(H)
&=\cos\Theta\prod_{i=1}^{3}\sqrt{1+\lambda_i^2},
\label{eq:phase-real}\\
\sigma_1(H)-\sigma_3(H)
&=\sin\Theta\prod_{i=1}^{3}\sqrt{1+\lambda_i^2}.
\label{eq:phase-imag}
\end{align}

The eigenvalues of $W$ are $\mu_i:=\cos\Theta\,\lambda_i-\sin\Theta$. 
Hence, by \eqref{eq:phase-real} and \eqref{eq:phase-imag}, 
\begin{align*}
&\sigma_3(W) -\sigma_1(W) \\
=&\cos^3 \Theta\, \sigma_3(H)-\cos^2\Theta\sin\Theta \,\sigma_2(H)
+\cos\Theta\sin^2\Theta\,\sigma_1(H)-\sin^3\Theta -\cos\Theta\,\sigma_1(H)+3\sin\Theta\\
=& -\cos^3\Theta \left( \sigma_1(H)-\sigma_3(H) \right) -\cos^2\Theta\sin\Theta \sigma_2(H) +3\sin\Theta-\sin^3\Theta \\
= &-\cos^3\Theta\sin\Theta\prod_{i=1}^{3}\sqrt{1+\lambda_i^2} -\cos^2\Theta\sin\Theta \Big( 1-\cos\Theta \prod_{i=1}^{3}\sqrt{1+\lambda_i^2} \Big)  +3\sin\Theta-\sin^3\Theta\\
=& 2\sin\Theta.
\end{align*}
This proves \eqref{eq:phase-det-trace}.

It remains to prove \eqref{eq:phase-cofactor}.
The eigenvalues of $I-\cof W$ are $1-\mu_j\mu_k$, where $\{i,j,k\}=\{1,2,3\}$. 
Since
\[
\lambda_\ell=\tan\theta_\ell, \qquad \theta_j+\theta_k=\Theta-\theta_i,
\]
we have
\[
1-\lambda_j\lambda_k
=1-\tan\theta_j\tan\theta_k
=\frac{\cos\theta_j\cos\theta_k-\sin\theta_j\sin\theta_k}
{\cos\theta_j\cos\theta_k}
=\frac{\cos(\theta_j+\theta_k)}
{\cos\theta_j\cos\theta_k}
\]
Similarly,
\[
\lambda_j+\lambda_k=\frac{\sin(\theta_j+\theta_k)}{\cos\theta_j\cos\theta_k}.
\]
Therefore,
\begin{align*}
1-\mu_j\mu_k &= 1- (\cos\Theta\,\lambda_j-\sin\Theta) (\cos\Theta\,\lambda_k-\sin\Theta)\\
&= \cos^2\Theta(1-\lambda_j\lambda_k) +\cos\Theta\sin\Theta(\lambda_j+\lambda_k)\\
&= \cos\Theta\, \frac{ \cos\Theta\cos(\theta_j+\theta_k) +\sin\Theta\sin(\theta_j+\theta_k) }{ \cos\theta_j\cos\theta_k }\\
&= \frac{ \cos\Theta\cos\theta_i }{ \cos\theta_j\cos\theta_k } = \frac{ 1}{ 1+\lambda_i^2 }\cos\Theta \displaystyle\prod_{m=1}^{3}\sqrt{1+\lambda_m^2} >0.
\end{align*}
Thus  
\[
I-\cof W>0.
\]
\end{proof}

For a matrix $A=(a_{ij})_{i,j=1}^3$, denote
\[
[a_{ij}]_{i,j=1}^2:=
\begin{pmatrix}a_{11}&a_{12}\\ a_{21}&a_{22}\end{pmatrix}
\]
for its upper-left $2\times2$ block.

\begin{lemma}
\label{lem:block-schur-cofactor}
Let
\[
A=\begin{pmatrix}B&\vec b\\ {\vec b}^{\mathsf T}&d\end{pmatrix}\in\operatorname{Sym}(3),
\qquad B\in\operatorname{Sym}(2),\quad \vec b\in\mathbb R^2,\quad d\ne0.
\]
Set $\vec{p}=\vec{b}/d$ and
\[
P:=\frac1{|d|}I_2-\frac d{|d|}B+|d|\vec{p}\vec{p}^{\mathsf T}.
\]
Then the upper-left $2\times2$ block of $I_3-\cof A$ satisfies
\begin{equation}
\label{eq:block-cofactor-identity}  \left[I_3-\cof A\right]_{{i,j=1}}^2  =|d|\,\cof P,
\end{equation} 
and 
\begin{equation}
\label{eq:block-determinant-identity} \det P=1+|\vec{p}|^2+\frac{\sigma_3(A)-\sigma_1(A)}d+\frac1{d^2}.
\end{equation}
In particular, $I_3-\cof A>0$ implies $P>0$.
\end{lemma}
\begin{proof}
Since $\vec b=d\vec{p}$, we have
\begin{equation} 
 \label{eq: B dp}
B-d\vec{p}\vec{p}^{\mathsf T}=\frac1dI_2-\frac d{|d|}P.
\end{equation}

We first prove \eqref{eq:block-cofactor-identity}. Denote
$B=(B_{ij})_{i,j=1}^2$ and $b=(b_1,b_2)^{\mathsf T}$. Expanding the
upper-left $2\times2$ block of $\cof A$ gives
\[
 [\cof A]_{i,j=1}^2 
=\begin{pmatrix}dB_{22}-b_2^2&b_1b_2-dB_{12}\\ b_1b_2-dB_{12}&dB_{11}-b_1^2\end{pmatrix}
=d\,\cof\bigl(B-d\vec{p}\vec{p}^{\mathsf T}\bigr).
\]
For every $Q\in \operatorname{Sym}(2)$, 
$\cof Q=(\sigma_1(Q))I_2-Q$. Hence the cofactor map is linear in
dimension two, and therefore
\[
d\,\cof\bigl(B-d\vec{p}\vec{p}^{\mathsf T}\bigr)
=d\,\cof\left(\frac1dI_2-\frac d{|d|}P\right)
=I_2-|d|\,\cof P.
\]
Thus the upper-left block of $I_3-\cof A$ is
$|d|\,\cof P$, which proves
\eqref{eq:block-cofactor-identity}.

We next prove \eqref{eq:block-determinant-identity}. Since
$B=d\vec{p}\vec{p}^{\mathsf T}+I_2/d-(d/|d|)P$, taking traces gives
\[
\sigma_1(A)=\sigma_1(B)+d
=\frac2d-\frac d{|d|}\sigma_1(P)+d|\vec{p}|^2+d.
\]
Subtracting $p_i$ times the third row of $A$ from its $i$th row for
$i=1,2$ and using \eqref{eq: B dp} give  
\begin{align*}
\sigma_3(A)
& =d\det\bigl(B-d\vec{p}\vec{p}^{\mathsf T}\bigr).
\\&=d\det\left(\frac1dI_2-\frac d{|d|}P\right)\\
&=d\left(\frac1{d^2}-\frac1{|d|}\sigma_1(P)+\det P\right)\\
&=\frac1d-\frac d{|d|}\sigma_1(P)+d\det P.
\end{align*} 
Therefore
\[
\sigma_3(A)-\sigma_1(A)=d\det P-\frac1d-d|\vec{p}|^2-d.
\]
Dividing by $d\ne0$ and rearranging proves
\eqref{eq:block-determinant-identity}.

Finally, suppose that $I_3-\cof A>0$. Its upper-left principal block is
positive definite, so \eqref{eq:block-cofactor-identity} gives $\cof P>0$.
If the eigenvalues of $P$ are $\lambda_1$ and $\lambda_2$, then the
eigenvalues of $\cof P$ are $\lambda_2$ and $\lambda_1$. Therefore
$\lambda_1,\lambda_2>0$, and hence $P>0$.
\end{proof}

We first establish the two-point estimate for $C^2$ solutions of \eqref{eq:main} for clarifying our ideas, and then extend it to viscosity solutions.

\begin{proposition}
\label{prop:two-point}
Let $u\in C^2(B_1)$ be a solution to \eqref{eq:main} with
$\osc_{B_1}u\le M$ with $|\Theta|<\pi/2$. Then, for any ball $B_\rho(z_0)\Subset B_1$ and any
$x,y\in\overline B_\rho(z_0)$ satisfying $|x-y|\le 2a/L$,
\begin{equation}
\label{eq:two-point} |u(y)-u(x)| \le \frac{1}{\cos\Theta} \left[ L|y-x|-\frac{L^2|y-x|^2}{4a} +\beta\left|\frac{x+y}{2}-z_0\right|^2 \right],
\end{equation}
where $a$ is a positive constant such that \( M\cos \Theta <a< 1\), $\beta $ is a positive constant such that $\frac{a}{\rho^2}<\beta<1$, and $L$ is  sufficiently large depending on $a, \beta, \Theta, \rho$.
\end{proposition}
\begin{proof}
Replacing $u$ by $-u$ if necessary,  we may restrict attention to the case $0\leq\Theta<\pi/2$. 
By the translation invariance of   \eqref{eq:main} and the symmetry of the right-hand side of \eqref{eq:two-point} in $x$ and $y$, it suffices to prove the estimate for $z_0=0$.

Define
\begin{equation}\label{eq:Z-def}
Z(x,y):= u(y)-u(x) -\frac{1}{\cos\Theta}
\left[
L|y-x|-\frac{L^2|y-x|^2}{4a}
+\beta\left|\frac{x+y}{2}\right|^2
\right]. 
\end{equation}

Since $Z$ is continuous, its maximum on this set is attained.
By \eqref{eq:Z-def}, estimate \eqref{eq:two-point} is equivalent to $Z\le0$.

\vspace{6pt}

\emph{Step 1:
We show that a positive
maximum cannot occur on the boundary or the diagonal.}

\vspace{8pt}
Indeed, if $|x-y|=2a/L$, then, since $\beta\ge0$,
\[
\begin{aligned}
Z(x,y) &\le M-\frac{1}{\cos\Theta} \left(L|y-x|-\frac{L^2|y-x|^2}{4a}\right) \le M-\frac{a}{\cos\Theta}<0.
\end{aligned}
\]

If $x\in\partial B_{\rho}$ or $y\in \partial B_{\rho}$, then
\[
\left|\frac{x+y}{2}\right| \ge \rho-\frac{|x-y|}{2} \ge \rho-\frac{a}{L}>0, 
\]
where $L$ is chosen sufficiently large so that $\rho-\frac{a}{L}>0$.

Since $|x-y|\le 2a/L$,
\[
L|y-x|-\frac{L^2|y-x|^2}{4a} = L|y-x|\left(1-\frac{L|y-x|}{4a}\right) \ge 0.
\]
Hence
\[
\begin{aligned}
Z(x,y) &\le M-\frac{\beta}{\cos\Theta} \left|\frac{x+y}{2}\right|^2 \le M-\frac{\beta}{\cos\Theta} \left(\rho-\frac{a}{L}\right)^2 <M-\frac{a}{\cos\Theta}<0,
\end{aligned}
\]
Here $L$ is chosen sufficiently large so that $\beta \left(\rho-\frac{a}{L}\right)^2>a$; later, we choose $\beta$, $a$, and $\rho$ so that $\beta\rho^2>a$.

If $x=y$, then 
\[
Z(x,x)= -\frac{\beta}{\cos\Theta}|x|^2\le0.
\]
Therefore, if $Z$ is positive somewhere, it attains a positive maximum at a pair of distinct interior points $x,y$.

\vspace{6pt}

\emph{Step 2:
We show that a positive
maximum cannot occur in the interior.}

\vspace{8pt}

At such an interior maximum point $(x,y)$, choose coordinates so that
\[
e_3=\frac{y-x}{|y-x|}.
\]
For $v=(v_1,v_2,v_3)\in\mathbb R^3$, denote $v_{1,2}:=(v_1,v_2)\in\mathbb R^2$.  Given $\xi,\eta\in\mathbb R^3$, consider the variation
\[
x(t)=x+t\xi,\qquad y(t)=y+t\eta.
\]
Since $Z$ attains a local maximum at $(x,y)$,
\[
\left.\frac{d^2}{dt^2}Z(x(t),y(t))\right|_{t=0}\le 0.
\]

We first note that
\[
\left. \frac{d^2}{dt^2} \bigl(u(y(t))-u(x(t))\bigr) \right|_{t=0} = \eta^T D^2u(y)\eta-\xi^T D^2u(x)\xi .
\]
Moreover,
\[
	 \left. \frac{d}{dt}|y(t)-x(t)| \right|_{t=0} = \eta_3-\xi_3,
	\quad  	\text{and}\quad
	 \left. \frac{d^2}{dt^2}|y(t)-x(t)| \right|_{t=0} = \frac{|(\eta-\xi)_{1,2}|^2}{|y-x|}.
 \]
Hence
\begin{align*}
    &\left. \frac{d^2}{dt^2} \left( L|y(t)-x(t)| -\frac{L^2}{4a}|y(t)-x(t)|^2 \right) \right|_{t=0} \\
    \qquad=& \left( \frac{L}{|y-x|}-\frac{L^2}{2a} \right) |(\eta-\xi)_{1,2}|^2 -\frac{L^2}{2a}(\eta_3-\xi_3)^2.
\end{align*}
Finally,
\[
\left. \frac{d^2}{dt^2} \left| \frac{x(t)+y(t)}2 \right|^2 \right|_{t=0} = \frac12|\xi+\eta|^2.
\]
Therefore
\begin{align}
    0\ge{}& \left. \frac{d^2}{dt^2}Z(x(t),y(t))
    \right|_{t=0} \notag\\
    ={}& \eta^TD^2u(y)\eta-\xi^TD^2u(x)\xi -\frac1{\cos\Theta} \left( \frac{L}{|y-x|}-\frac{L^2}{2a} \right) |(\eta-\xi)_{1,2}|^2\nonumber\\
    &+\frac1{\cos\Theta} \frac{L^2}{2a}(\eta_3-\xi_3)^2 -\frac1{\cos\Theta} \frac{\beta}{ 2 } |\xi+\eta|^2 . \label{eq:second-variation-Z}
\end{align}

Define
\begin{equation}\label{eq:w-def}
w:=\cos\Theta\,u-\frac{\sin\Theta}{2}|x|^2.
\end{equation}
Then
\begin{equation}\label{eq:D2w}
D^2w=\cos\Theta D^2u-\sin\Theta I.
\end{equation}
Since $\cos\Theta>0$, multiplying  \eqref{eq:second-variation-Z} by $\cos\Theta$ and using \eqref{eq:D2w} yields
\begin{align}\label{eq:joint-second-derivative}
&\eta^{\mathsf T}D^2w(y)\eta
-\xi^{\mathsf T}D^2w(x)\xi
+\sin\Theta\bigl(|\eta|^2-|\xi|^2\bigr)
\notag\\
\le& 
\left(
\frac{L}{|y-x|}-\frac{L^2}{2a}
\right)
|(\eta-\xi)_{1,2}|^2
-\frac{L^2}{2a}(\eta_3-\xi_3)^2
+\frac{\beta}{2}|\xi+\eta|^2.
\end{align}

Taking $(\xi,\eta)=(e_3,0)$  in \eqref{eq:joint-second-derivative}, we obtain
 \[
 -w_{33}(x)-\sin\Theta \le -\frac{L^2}{2a}+\frac{\beta}{2},
 \]
 and hence, by choosing $L$ sufficiently large such that $$\frac{L^2}{2a}-\sin\Theta-\frac{\beta}{2} >\frac{L^2}{4a}>2 \sin \Theta,$$
 we have 
 \begin{equation}
 \label{eq:dx-lower} w_{33}(x) \ge \frac{L^2}{2a}-\sin\Theta-\frac{\beta}{2}>\frac{L^2}{4a} 
 \end{equation}
Choosing $(\xi,\eta)=(0,e_3)$ yields
 \[
 w_{33}(y)+\sin\Theta \le -\frac{L^2}{2a}+\frac{\beta}{2},
 \]
 so that
 \begin{equation}
 \label{eq:dy-lower} 
 -w_{33}(y) \ge \frac{L^2}{2a}+\sin\Theta-\frac{\beta}{2}>\frac{L^2}{4a}.
 \end{equation}

We next choose variations that eliminate the mixed entries
$w_{13}$ and $w_{23}$. Define
\begin{equation}
\label{eq:p-def} \vec{p} :=\frac{1}{w_{33} }\bigl(w_{13} ,w_{23} \bigr).
\end{equation}
For $v=(v_1,v_2)\in\R^2$, take
\begin{equation}\label{eq:lifted-vectors}
\xi=(v_1,v_2,-\vec{p}(x)\cdot v),
\qquad
\eta=(v_1,v_2,-\vec{p}(y)\cdot v).
\end{equation}
  Define
\begin{equation}
\label{eq:P-def}
P:=\frac1{|w_{33}|}I_2-\frac{w_{33}}{|w_{33}|}[w_{ij}]_{i,j=1}^2+|w_{33}|\vec p{\vec p}^{\mathsf T}.
\end{equation}

The definition \eqref{eq:p-def} gives
$w_{i3}(x)=w_{33}(x)(\vec{p}(x))_i$ for $i=1,2$. Therefore
\begin{align}
\xi^{\mathsf T}D^2w(x)\xi &= v^{\mathsf T} \left( [w_{ij}(x)]_{i,j=1}^2-w_{33}(x)\vec{p}(x)\vec{p}(x)^{\mathsf T} \right)v \notag\\
&= v^{\mathsf T} \left( \frac1{w_{33}(x)}I_2-P(x) \right)v. \label{eq:w-x-lift}
\end{align}
Similarly,
\begin{align}
\eta^{\mathsf T}D^2w(y)\eta &= v^{\mathsf T} \left( [w_{ij}(y)]_{i,j=1}^2-w_{33}(y)\vec{p}(y)\vec{p}(y)^{\mathsf T} \right)v \notag\\
&= v^{\mathsf T} \left( P(y)+\frac1{w_{33}(y)}I_2 \right)v. \label{eq:w-y-lift}
\end{align}

By the choice of $\eta$ and $\xi$, we have
\begin{align*}
|\eta|^2-|\xi|^2 &=(\vec{p}(y)\cdot v)^2-(\vec{p}(x)\cdot v)^2,\\
|(\eta-\xi)_{1,2}|^2 &=0,\,\, (\eta_3-\xi_3)^2 =((\vec{p}(y)-\vec{p}(x))\cdot v)^2,\\
|\xi+\eta|^2 &=4|v|^2+((\vec{p}(x)+\vec{p}(y))\cdot v)^2.
\end{align*}
Substituting these identities into
\eqref{eq:joint-second-derivative} gives
\begin{align}
&v^{\mathsf T}(P(x)+P(y))v +\left(\frac1{w_{33}(y)}-\frac1{w_{33}(x)}\right)|v|^2 +\sin\Theta \bigl((\vec{p}(y)\cdot v)^2-(\vec{p}(x)\cdot v)^2\bigr) \notag\\
\le& -\frac{L^2}{2a} ((\vec{p}(y)-\vec{p}(x))\cdot v)^2 +\frac{\beta}{2} \left[ 4|v|^2+((\vec{p}(x)+\vec{p}(y))\cdot v)^2 \right]. \label{eq:lifted-inequality}
\end{align}

Set
\begin{equation}
\label{eq:q-def} \vec q:=\frac{\vec{p}(x)+\vec{p}(y)}{2}.
\end{equation}
Using
\[
(\vec{p}(y)\cdot v)^2-(\vec{p}(x)\cdot v)^2=2(\vec q\cdot v)\bigl((\vec{p}(y)-\vec{p}(x))\cdot v\bigr)
\]
and $\vec{p}(x)+\vec{p}(y)=2\vec q$, the last term on the right-hand side of \eqref{eq:lifted-inequality} becomes
\[
\frac{\beta}{2}\left[4|v|^2+((\vec{p}(x)+\vec{p}(y))\cdot v)^2\right]=2\beta|v|^2+2\beta(\vec q\cdot v)^2.
\]
Rearranging \eqref{eq:lifted-inequality}, we obtain
\begin{align*}
v^{\mathsf T}(P(x)+P(y))v \le{}& \left(2\beta+\frac1{w_{33}(x)}-\frac1{w_{33}(y)}\right)|v|^2+2\beta(q\cdot v)^2\\
&-\left[\frac{L^2}{2a}\bigl((\vec{p}(y)-\vec{p}(x))\cdot v\bigr)^2+\sin\Theta\bigl((\vec{p}(y)\cdot v)^2-(\vec{p}(x)\cdot v)^2\bigr)\right].
\end{align*}
We complete the square in the bracketed expression as follows:
\begin{align}
&\frac{L^2}{2a} ((\vec{p}(y)-\vec{p}(x))\cdot v)^2 +\sin\Theta \bigl((\vec{p}(y)\cdot v)^2-(\vec{p}(x)\cdot v)^2\bigr) \notag\\
 = &\frac{L^2}{2a} \left( (\vec{p}(y)-\vec{p}(x))\cdot v +\frac{\sin\Theta}{L^2/(2a)}\vec{q}\cdot v \right)^2 -\frac{\sin^2\Theta}{L^2/(2a)} (\vec{q}\cdot v)^2. \label{eq:complete-square}
\end{align}
Substituting \eqref{eq:complete-square} into the preceding inequality gives
\begin{equation}
\label{eq:PxPy-upper-quadratic}
\begin{aligned}
	v^{\mathsf T}(P(x)+P(y))v \le{}& \left(2\beta+\frac1{w_{33}(x)}-\frac1{w_{33}(y)}\right)|v|^2+\left(2\beta+\frac{\sin^2\Theta}{L^2/(2a)}\right)(\vec{q}\cdot v)^2\\
	&-\frac{L^2}{2a}\left((\vec{p}(y)-\vec{p}(x))\cdot v+\frac{\sin\Theta}{L^2/(2a)}\vec{q}\cdot v\right)^2 \\	
  \le& \left(2\beta+\frac1{w_{33}(x)}-\frac1{w_{33}(y)}\right)|v|^2+\left(2\beta+\frac{\sin^2\Theta}{L^2/(2a)}\right)(\vec{q}\cdot v)^2.
\end{aligned}
\end{equation}
By \eqref{eq:dy-lower} and \eqref{eq:dx-lower},
\[
\frac1{w_{33}(x)}-\frac1{w_{33}(y)} \leq\frac{8a}{L^2}.
\]
Therefore \eqref{eq:PxPy-upper-quadratic} gives
\begin{equation}
	\label{eq:PxPy-matrix-upper} P(x)+P(y)\leq\left(2\beta+\frac{8a}{L^2}\right)(I_2+\vec{q}\vec{q}^{\mathsf T}).
\end{equation}

We next use the algebraic identities satisfied by $D^2w$.
By \eqref{eq:D2w} and
Lemma~\ref{lem:phase-identity},
\begin{equation}
\label{eq:w-phase-identities} 
 \sigma_3(D^2w(z))-\sigma_1(D^2w(z))=2\sin\Theta, \qquad I_3-\cof D^2w(z)>0, \qquad z\in\{x,y\}.
\end{equation}

For $A=D^2w(z)$, the entry $d$ in Lemma~\ref{lem:block-schur-cofactor} is $w_{33}(z)$. By \eqref{eq:dx-lower} and \eqref{eq:dy-lower}, $w_{33}(z)\neq0$ for $z\in\{x,y\}$, so Lemma~\ref{lem:block-schur-cofactor} applies at both points. By \eqref{eq:P-def}, the corresponding matrix $P$ is precisely $P(z)$. 

Together with \eqref{eq:w-phase-identities}, Lemma~\ref{lem:block-schur-cofactor} gives $P(z)>0$ and
\begin{equation}
\label{eq:det-Pz} \det P(z)=1+|p(z)|^2+\frac{2\sin\Theta}{w_{33}(z)}+\frac1{w_{33}(z)^2}, \qquad z\in\{x,y\}.
\end{equation}

Since $P(x)$ and $P(y)$ are positive definite, taking determinants of both sides of \eqref{eq:PxPy-matrix-upper} gives
\begin{align}
\label{eq:det-upper} 
\sqrt{\det(P(x)+P(y))} &\le \left( 2\beta+\frac{8a}{L^2} \right) \sqrt{\det(I_2+\vec{q}\vec{q}^{\mathsf T})}\\
&\le 2\left( \beta+\frac{4a}{L^2} \right) \sqrt{1+|\vec{q}|^2},
\end{align}
where \(\det(I_2+\vec{q}\vec{q}^{\mathsf T})=1+|\vec{q}|^2.\)

We next derive the lower bound.
Since $w_{33}(x)>0$, \eqref{eq:det-Pz} with $z=x$ gives
\begin{equation}
\label{eq:detPx-lower} \sqrt{\det P(x)} \ge \sqrt{1+|\vec{p}(x)|^2}.
\end{equation}
Taking $z=y$ in \eqref{eq:det-Pz} and using
\eqref{eq:dy-lower}, we obtain
\[
\det P(y)
=1+|\vec{p}(y)|^2-\frac{2\sin\Theta}{-w_{33}(y)}
+\frac1{w_{33}(y)^2}
\ge 1+|\vec{p}(y)|^2-\frac{8a\sin\Theta}{L^2}.
\]

Hence
\begin{equation}
\label{eq:detPy2} \sqrt{\det P(y)} \ge \sqrt{1+|\vec{p}(y)|^2} -\frac{8a\sin\Theta}{L^2},
\end{equation}
because  
\begin{align*}
 &\sqrt{1+|\vec{p}(y)|^2}-\sqrt{1+|\vec{p}(y)|^2-\frac{8a\sin\Theta}{L^2}} 
\le\frac{8a\sin\Theta}{L^2}.
\end{align*}
Since $P(x),P(y)>0$, the concavity of $\det^{1/2}$, together with \eqref{eq:detPx-lower} and \eqref{eq:detPy2}, gives
\begin{align}
\sqrt{\det(P(x)+P(y))} &\ge \sqrt{\det P(x)}+\sqrt{\det P(y)}\\
&\ge \sqrt{1+|\vec{p}(x)|^2} +\sqrt{1+|\vec{p}(y)|^2} -\frac{8a\sin\Theta}{L^2}\\
&\ge 2\sqrt{1+\left|\frac{\vec{p}(x)+\vec{p}(y)}{2}\right|^2}-\frac{8a\sin\Theta}{L^2}\\
&= 2\sqrt{1+|\vec{q}|^2}-\frac{8a\sin\Theta}{L^2}.
\end{align}

Therefore
\begin{align}
\sqrt{\det(P(x)+P(y))} \ge 2\left( 1-\frac{8a\sin\Theta}{L^2} \right) \sqrt{1+|\vec q|^2}. \label{eq:det-lower}
\end{align}

Comparing \eqref{eq:det-upper} and
\eqref{eq:det-lower} gives
\[
2\left( \beta+\frac{4a}{L^2} \right) \ge  2\left( 1-\frac{8a\sin\Theta}{L^2} \right).
\]
That is,
\[
1-\beta \le \frac{4a} {L^2}\bigl(1+2\sin\Theta\bigr).
\] 
This contradicts the choice $\beta<1$ when $L$ is sufficiently large. Hence $Z$ cannot have a positive
interior maximum.
\end{proof}

\begin{remark}\label{remark}
    One may consider the two-point auxiliary function
\[
Z(x,y)=u(x)-u(y)-\phi(|x-y|)-a(|x|^2+|y|^2),
\]
where $\phi$ is chosen appropriately in the proof and $a>0$ is a constant. 
Auxiliary functions of this form have been used in the theory of viscosity
solutions; see, for example, \cite{CrandallIshiiLions}.
The auxiliary function \eqref{eq:Z-def} used in this paper can also be
written in this form by the identity
\[
|x|^2+|y|^2
=\frac12|x-y|^2+\frac12|x+y|^2.
\]
\end{remark}

Before treating the viscosity case, we recall a classical result of
Crandall--Ishii--Lions in the form needed below. For the reader's convenience, we just write the formula that we need. For $x_0\in\Omega$, let $J^{\,2,-}u(x_0)$ denote the set of all $(p,X)\in\mathbb R^n\times\operatorname{Sym}(n)$ such that
\[
u(z)\ge u(x_0)+p\cdot(z-x_0)+\frac12(z-x_0)^{\mathsf T}X(z-x_0)+o(|z-x_0|^2)
\qquad\text{as }z\to x_0.
\]
Similarly, $J^{\,2,+}u(x_0)$ denotes the set of all $(p,X)\in\mathbb R^n\times\operatorname{Sym}(n)$ such that
\[
u(z)\le u(x_0)+p\cdot(z-x_0)+\frac12(z-x_0)^{\mathsf T}X(z-x_0)+o(|z-x_0|^2)
\qquad\text{as }z\to x_0.
\]
Finally, $(p,X)\in\overline J^{\,2,\pm}u(x_0)$ if there exist $x_k\to x_0$ and $(p_k,X_k)\to(p,X)$ such that $(p_k,X_k)\in J^{\,2,\pm}u(x_k)$. 

\begin{lemma}[Theorem 3.2 in \cite{CrandallIshiiLions}]
\label{lem:theorem-of-sums}
Let $\Omega\subset\mathbb R^n$, $u\in C(\Omega)$, and
$\Phi\in C^2(\Omega\times\Omega)$. Assume that
\[
Z(x,y)=u(y)-u(x)-\Phi(x,y)
\]
has a local maximum at $(\bar x,\bar y)\in\Omega\times\Omega$.

Set
\[
A=D^2_{(x,y)}\Phi(\bar x,\bar y).
\]
Then, for every $\delta>0$, there exist symmetric matrices
$X,Y\in\operatorname{Sym}(n)$ such that
\begin{equation}\label{eq:theorem-sums-inclusions}
\bigl(-D_x\Phi(\bar x,\bar y),X\bigr)\in\overline J^{\,2,-}u(\bar x), \qquad \bigl(D_y\Phi(\bar x,\bar y),Y\bigr)\in\overline J^{\,2,+}u(\bar y).
\end{equation}
Moreover,
\begin{equation}\label{eq:theorem-sums-matrix}
- \left(\frac1\delta+\|A\|\right)I_{2n}
\le
\begin{pmatrix}
-X&0\\
0&Y
\end{pmatrix}
\le
A+\delta A^2.
\end{equation}
\end{lemma}

\medskip

\begin{proof}[Proof of Theorem \ref{thm:small-height}]
{
Replacing $u$ by $-u$ if necessary, we may assume
$0\leq\Theta<\pi/2$. By the scaling $u(x)\mapsto u(Rx)/R^2$, it suffices to prove the estimate in $B_1$.

Set
\[
a=\frac{1+M\cos\Theta}{2},
\qquad
\rho=\sqrt{\frac{1+a}{2}}
=\sqrt{\frac{3+M\cos\Theta}{4}},
\]
and choose $a/\rho^2<\beta<1$. 
We then take $L$ sufficiently large so that the hypotheses of
Proposition~\ref{prop:two-point} are satisfied and
\[
1-\beta>
\frac{1+\sin\Theta}
{L^2/(2a)-\sin\Theta-\beta/2}.
\]
}

We begin by proving Proposition~\ref{prop:two-point} for viscosity solutions.   Since the two-point estimate \eqref{eq:two-point} has already been proved for $C^2$ solutions, we describe only the modifications required in the viscosity setting.
 
Suppose, to the contrary, that $Z$ attains a positive
interior maximum at two distinct points $x,y$. For every $\varepsilon>0$,
applying Lemma \ref{lem:theorem-of-sums} to \eqref{eq:Z-def} with
parameter $\delta=\varepsilon\cos\Theta$, there exist
$A_x,A_y\in\operatorname{Sym}(3)$ such that
\[
\sum_{i=1}^{3}\arctan\lambda_i(A_x)\leq\Theta\leq\sum_{i=1}^{3}\arctan\lambda_i(A_y).
\]
Since $t\mapsto\sum_{i=1}^{3}\arctan\lambda_i(A+tI)$ is strictly increasing in $t$, there exist $t_x,t_y\geq0$ such that
\[
\sum_{i=1}^{3}\arctan\lambda_i(A_x+t_xI)=\Theta=\sum_{i=1}^{3}\arctan\lambda_i(A_y-t_yI).
\]
Since $t_x,t_y\geq0$, for every $\xi,\eta\in\mathbb R^3$,
\[
\eta^{\mathsf T}(A_y-t_yI)\eta-\xi^{\mathsf T}(A_x+t_xI)\xi
\leq\eta^{\mathsf T}A_y\eta-\xi^{\mathsf T}A_x\xi.
\]
Replace $A_x$ by $A_x+t_xI$ and $A_y$ by $A_y-t_yI$. Then
\[
\sum_{i=1}^{3}\arctan\lambda_i(A_x)=\Theta=\sum_{i=1}^{3}\arctan\lambda_i(A_y).
\]

Set $W_x:=\cos\Theta\,A_x-\sin\Theta I$ and $W_y:=\cos\Theta\,A_y-\sin\Theta I$, and define
\[
\beta_\varepsilon:=\beta+\varepsilon\beta^2,\qquad
N_\varepsilon:=\frac{L^2}{2a}-\sin\Theta-\frac{\beta}{2}-\varepsilon\left\{2\left(\frac{L^2}{2a}\right)^2+\frac{\beta^2}{2}\right\}.
\]
After multiplying by $\cos\Theta$ and using the definitions of $W_x$ and $W_y$, Lemma~\ref{lem:theorem-of-sums} gives, for every
$\xi,\eta\in\mathbb R^3$ satisfying $(\eta-\xi)_{1,2}=0$,
\[
\begin{aligned}
	&\eta^{\mathsf T}W_y\eta-\xi^{\mathsf T}W_x\xi+\sin\Theta\bigl(|\eta|^2-|\xi|^2\bigr)\\
	&\quad\leq-\frac{L^2}{2a}(\eta_3-\xi_3)^2+\frac{\beta}{2}|\xi+\eta|^2
	+\varepsilon\left\{2\left(\frac{L^2}{2a}\right)^2(\eta_3-\xi_3)^2+\frac{\beta^2}{2}|\xi+\eta|^2\right\}\\
	&\quad=-\left(N_\varepsilon+\sin\Theta+\frac{\beta_\varepsilon}{2}\right)(\eta_3-\xi_3)^2+\frac{\beta_\varepsilon}{2}|\xi+\eta|^2.
\end{aligned}
\]
Taking $(\xi,\eta)=(e_3,0)$ and $(\xi,\eta)=(0,e_3)$ gives
\[
(W_x)_{33}\geq N_\varepsilon,\qquad -(W_y)_{33}\geq N_\varepsilon+2\sin\Theta\geq N_\varepsilon.
\]
The phase identities in Lemma~\ref{lem:phase-identity} apply to both $W_x$ and $W_y$.

Since $\beta_\varepsilon\to\beta$ and $N_\varepsilon\to \frac{L^2}{2a}-\sin\Theta-\frac{\beta}{2}$ as $\varepsilon\downarrow0$, choose $\varepsilon>0$ sufficiently small that
\begin{equation}\label{condition}
N_\varepsilon>2\sin\Theta,\qquad
1-\beta_\varepsilon>\frac{1+\sin\Theta}{N_\varepsilon}.
\end{equation}

Define $\vec p(z)$, $P(z)$, and $\vec q$ as in the smooth case, with
$D^2w(z)$ replaced by $W_z$. The same determinant comparison gives
\[
\begin{aligned}
\sqrt{\det(P(x)+P(y))}
&\leq2\left(\beta_\varepsilon+\frac1{N_\varepsilon}\right)
\sqrt{1+|\vec q|^2},\\
\sqrt{\det(P(x)+P(y))}
&\geq2\left(1-\frac{\sin\Theta}{N_\varepsilon}\right)
\sqrt{1+|\vec q|^2}.
\end{aligned}
\]
Therefore
\[
1-\beta_\varepsilon\leq
\frac{1+\sin\Theta}{N_\varepsilon},
\]
which contradicts the second inequality in \eqref{condition}. Hence $Z\leq0$. The same argument applies with $|(x+y)/2|^2$ replaced by $|(x+y)/2-z_0|^2$.

Take $x,y\in B_{(1-\rho)/2}$. Since $1/2<\rho<1$, we have $B_\rho(x)\Subset B_1$ and $x,y\in B_\rho(x)$. If $|x-y|\leq2a/L$, applying Proposition~\ref{prop:two-point} with $z_0=x$, we obtain
\[
\begin{aligned}
|u(y)-u(x)|
&\leq\frac1{\cos\Theta}\left(L|y-x|-\frac{L^2|y-x|^2}{4a}+\beta\left|\frac{x+y}{2}-x\right|^2\right)\\
&=\frac1{\cos\Theta}\left(L|y-x|-\frac{L^2|y-x|^2}{4a}+\frac{\beta|y-x|^2}{4}\right)\\
&\leq\frac1{\cos\Theta}\left(L|y-x|+\frac{\beta}{4}\frac{2a}{L}|y-x|\right)\\
&=\frac1{\cos\Theta}\left(L+\frac{a\beta}{2L}\right)|y-x|.
\end{aligned}
\]
If $|x-y|>2a/L$, then $L|x-y|/(2a)>1$. Since $M\geq0$, the oscillation bound gives
\[
|u(x)-u(y)|\leq\osc_{B_1}u\leq M\leq\frac{ML}{2a}|x-y|.
\]
Therefore
\[
  \Lip_{B_{(1-\rho)/2}}u\le C_0,
  \qquad
  C_0=\max\left\{
    \frac{L+a\beta/(2L)}{\cos\Theta},\frac{ML}{2a}
  \right\}.
\] 
Here $C_0$ depends only on $\Theta$ and $M$, and  $\rho$ is a positive constant depending on $a$ and $\beta$ such that $\frac{a}{\rho^2}<\beta<1.$  
  
\end{proof}

\begin{remark}
For $n\geq4$ and each fixed subcritical phase
$|\Theta|<(n-2)\pi/2$,
it is natural to define the scale-invariant critical height by
\[
M_{\mathrm{crit}}(n,|\Theta|)
:=
\sup\Bigl\{
m>0:\ 
\text{for every }0<M<m,\ 
\osc_{B_1}u\leq M
\Longrightarrow
\Lip_{B_{\rho }}u\leq C
\Bigr\}.
\]
In fact, one can use the one-point maximum principle to see that in the range
$(n-3)\pi/2<|\Theta|<(n-2)\pi/2$, $\osc_{B_R}u\leq MR^2$ implies
$\Lip_{B_{\rho R}}u\leq CR$ whenever
$0<M<\frac12\cot\bigl((n-2)\pi/2-|\Theta|\bigr)$, where $C>0$ and
$0<\rho<1$ depend only on $n$, $|\Theta|$, and $M$.
Our two-point method suggests an
analogous uniform interior Lipschitz estimate under a suitable
small-height assumption.

At $\Theta=0$, the higher-dimensional calculation suggests that
the estimate should hold whenever
\[
\osc_{B_R}u\leq MR^2,
\qquad
M<\tan\frac{\pi}{2(n-1)}.
\]
In particular, this gives the expected sufficient bounds
$M<1/\sqrt3$ in dimension four and $M<\sqrt2-1$ in dimension five.

We expect these bounds to yield uniform interior Lipschitz estimates
for continuous viscosity solutions, with the constants and the
interior radius depending only on $n$, $|\Theta|$, and $M$.
Their sharpness is unknown.
A natural problem is to establish this higher-dimensional extension
and to determine the optimal height threshold
$M_{\mathrm{crit}}(n,\Theta)$ for each subcritical phase.
\end{remark}

\begin{remark}
The argument also suggests investigating small-height
Lipschitz estimates for other nonconcave equations of the form
\[
F(D^2u)=\sum_{i=1}^{n}a_i\sigma_i(D^2u)=f(x).
\]
What additional conditions on the coefficients, the elliptic
branch, and the right-hand side ensure that a sufficiently
small oscillation yields an interior Lipschitz estimate?
It would be interesting to identify analogues of the
algebraic identities and positivity properties in
Lemma~\ref{lem:phase-identity} that permit a two-point
comparison for these equations.
\end{remark}

\medskip 
\noindent{\it \bf Acknowledgments.}

\
 Wei Wei would like to thank  Prof. Guohuan Qiu for stimulating discussions on this topic. 

\

\noindent{\it \bf On the use of AI.}

We first investigated the optimality of the threshold $\frac{1}{2}\tan\Theta$ obtained from the one-point maximum principle by seeking counterexamples. ChatGPT 5.6 Pro was used as a computational aid in this exploration. The examples obtained in this way reached only the threshold $M=\sec\Theta$. Motivated by this observation, the authors then explored whether the height restriction could be improved to $M<\sec\Theta$ by modifying the auxiliary function \eqref{auxiliary function1}. The resulting calculations, again assisted by ChatGPT 5.6 Pro, showed that the favorable cutoff term was insufficient to control the unfavorable terms.

The two-point auxiliary function in Remark~\ref{remark}, as well as the decision to pursue a two-point maximum-principle argument , was proposed by the authors. ChatGPT 5.6 Pro was then used to carry out and organize the corresponding calculations and algebraic estimates. More generally, ChatGPT 5.6 Pro was used only as an aid for exploratory computations, algebraic manipulations, and the organization of intermediate estimates.

The authors independently verified all arguments and take full responsibility for the completeness and correctness of the proofs.

\bibliography{special_lagrangian_references.bib}
\bibliographystyle{amsalpha}

\end{document}